\documentclass[12pt]{amsart}

\usepackage{amsthm, mathrsfs,amssymb,amsmath}
\usepackage{enumerate}
\usepackage{xcolor}

\makeatletter
\@namedef{subjclassname@2010}{%
\textup{2010} Mathematics Subject Classification}
\makeatother

\allowdisplaybreaks

\newtheorem{thm}{Theorem}[section]
\newtheorem{lem}[thm]{Lemma}
\newtheorem{prop}[thm]{Proposition}
\newtheorem{cor}[thm]{Corollary}

\theoremstyle{definition}
\newtheorem{defn}[thm]{Definition} 

\newtheorem{exmp}{Example}[section]

\theoremstyle{remark}
\newtheorem{rem}{Remark}

\newcommand{\spt}{\textnormal{spt}}

\title[Ramsey theory for hypergroups]{Ramsey theory for hypergroups}

\author{Vishvesh Kumar}
\address{Vishvesh Kumar \endgraf Department of Mathematics \endgraf Indian Institute of Technology Delhi \endgraf New Delhi - 110 016, India.} 
\email{vishveshmishra@gmail.com}
\author{Kenneth A. Ross} 
\address{ Kenneth A. Ross, \endgraf Prof Emeritus  \endgraf Department of Mathematics \endgraf University of Oregon \endgraf  Eugene, OR 97403, USA.} 
\email{kenross.math@gmail.com}
\author{Ajit Iqbal Singh}
\address{Ajit Iqbal Singh \endgraf INSA Emeritus Scientist\endgraf The Indian National Science Academy \endgraf New Delhi - 110002, India.} 
\email{ajitis@gmail.com}

\begin{document}

\begin{abstract}  In this paper, Ramsey theory for discrete hypergroups is introduced with  emphasis on polynomial hypergroups, discrete orbit hypergroups  and  hypergroup deformations of semigroups. In this context, new notions of Ramsey principle for hypergroups and $\alpha$-Ramsey hypergroup, $0 \leq \alpha<1,$  are defined and studied. 

\end{abstract}
\maketitle

\section{Introduction} Ramsey theory \cite{Ramsey1}, now a well-developed branch of combinatorics, has a long history dating back to 1892 starting with David Hilbert \cite{Hilbert}. For a discrete semigroup $(S, \cdot),$ the algebra structure of Stone-$\check{\mbox{C}}$ech compactification $\beta S$ of $S$ has been utilized with a great advantage to study the Ramsey theory;  a good account of all this can be found in Hindman and Strauss \cite{Dona}. 

Dunkl \cite{Dunkl}, Jewett \cite{Jewett} and Spector \cite{Spector} independently created locally compact hypergroups  under different names with the purpose of doing standard harmonic analysis. Hypergroups, same as convolution spaces, in short, convos, or, semi convolution spaces, in short, semiconvos \cite{Jewett} are probabilistic generalization of groups or semigroups respectively. In \cite{KKA}, we presented a necessary and sufficient condition for a discrete semigroup to become a semiconvo or hypergroup after deforming the product at idempotents. This together with polynomial hypergroups served as the initial motivation for this work. Another motivation for this came from the development of configurations and configuration equations in the hypergroup setting \cite{Willson}.
A paper of Golan and Tsaban \cite{Gili} gave a few more analogues or variants of Ramsey principle on groups and semigroups.

In this paper, we define the concepts of Ramsey hypergroups, almost-Ramsey hypergroups and almost-strong Ramsey hypergroups and prove that the double coset  hypergroup  of a commutative discrete hypergroup $K$ by a finite subgroup of the centre $Z(K)$ of $K,$ defined and studied by Ross \cite{Ken}, is a Ramsey hypergroup or almost-Ramsey hypergroup if $K$ is so. We  also prove that no polynomial hypergroup is an almost-strong Ramsey hypergroup. Further, we show that Chebyshev polynomial hypergroup of second kind is not an almost-Ramsey hypergroup.  Next, we introduce variants like $\alpha$-Ramsey semiconvos and hypergroups. We prove that every polynomial hypergroup is a $0$-Ramsey hypergroup. We show that the Chebyshev polynomial hypergroup of second kind can never be an $\alpha$-Ramsey hypergroup for $0<\alpha<1.$ We also prove that orbit space hypergroups arising from  group actions on discrete groups are $0$-Ramsey semiconvos or hypergroups. Finally, we generalize this to the semigroup setting for the action of a finite group of automorphisms of a semigroup.

For the sake of convenience we give relevant basics of hypergroups and Ramsey theory for semigroups in the next section. Section 3 contains our main results on hypergroup versions of the Ramsey principles and includes motivation with examples. The final section 4 is devoted to the variants like $\alpha$-Ramsey hypergroup, $0 \leq \alpha<1$.  

Let $\mathbb{Z}_+= \mathbb{N} \cup \{0\}$ and, $ \times$ and $+$ be the usual multiplication and addition respectively. For a set $A,$ the power set of $A$ will be denoted by $\mathcal{P}(A).$ For notational convenience, we take empty sums to be zero and empty product to be one.
\section{Relevant Basics of  Ramsey theory and hypergroups}
We give some basics of Ramsey theory and hypergroups which are relevant to our study.
\subsection{Basics of Ramsey theory for semigroups} \label{rs} 

To begin with we confine our attention to the following aspects of Ramsey theory, viz., Theorem \ref{fsum}, Theorem \ref{finitepro} below which served as a motivation for this work. We refer to \cite{Hindman} and \cite{Dona} or any other suitable sources like \cite{Gili} for more details.

Given a sequence $\langle x_n\rangle_{n=1}^\infty$ in $\mathbb{N},$\, $$\text{FS}(\langle x_n\rangle_{n=1}^\infty)= \{ \sum_{n \in F}x_n: F\, \text{is a  non-empty finite subset of}\,\, \mathbb{N} \}.$$   
\begin{thm}  {\bf $\text{Hindman}$} \cite[ Theorem 3.1 (Finite Sums Theorem)]{Hindman} \label{fsum} Let $r \in \mathbb{N}$ and let $\mathbb{N} = \bigcup_{i=1}^r A_i$ be a partition of $\mathbb{N}.$ There exist $i \in \{1,2,...,r\}$ and a sequence $\langle x_n\rangle_{n=1}^\infty$ in $\mathbb{N}$ such that FS$(\langle x_n\rangle_{n=1}^\infty) \subset A_i.$ 
\end{thm}
Any partition of a non-empty set $X = \cup_{i=1}^r X_i$ is usually called a {\it finite colouring} (or, in short, a {\it  colouring}, at times) of $X;$ and $X$ is said to be a coloured set.  A subset $A$ of a coloured set $X$ is  {\it monochromatic} if all members of $A$ have the same colour. $A$ is {\it almost-monochromatic} if all but finitely many members of $A$ have the same colour.

  The Finite Sums Theorem is often also stated as: for any colouring of  $\mathbb{N}$ there exists a sequence $\langle x_n\rangle_{n=1}^\infty$ in $\mathbb{N}$ such that all finite sums FS$(\langle x_n\rangle_{n=1}^\infty)$ are monochromatic.\\

Let $(S, \cdot)$ be a semigroup and $\langle x_n\rangle_{n=1}^\infty$ a sequence in $S.$ If $S$ is commutative, we set \, FP$(\langle x_n\rangle_{n=1}^\infty)= \{ \prod_{n \in F}x_n: F\, \text{is non-empty finite subset of}\,\, \mathbb{N} \}.$ If $S$ is non-commutative then FP$(\langle x_n\rangle_{n=1}^\infty)= \{ \prod_{j=1}^k x_{n_j}: k \in \mathbb{N},\, 1 \leq n_1< n_2< \cdots < n_k\}.$

Now, we replace $(\mathbb{N},+)$ by $(S, \cdot)$ and finite sums (FS) by finite products (FP) in Theorem \ref{fsum} above. Of course, if $(S, \cdot)$ has an idempotent $s$ then by taking the sequence $\langle x_n \rangle_{n=1}^\infty$ with $x_n=s$ for all $n \in \mathbb{N}$ and the set $A_i$ from the partition containing $s,$ the desired analogue of Theorem \ref{fsum} becomes immediately available. 

\begin{thm} \cite[Corollary 5.9]{Dona} \label{finitepro} Let $S$ be a semigroup, let $r \in \mathbb{N}$ and let $S = \bigcup_{i=1}^r A_i$ be a partition of $S.$ There exist $i \in \{1,2, \ldots,r\}$ and a sequence $\langle x_n\rangle_{n=1}^\infty$ in $S$ such that FP$(\langle x_n\rangle_{n=1}^\infty) \subset A_i.$
\end{thm}
It is implicit in the proof in \cite{Hindman} of Theorem \ref{fsum} above that a  sequence $\langle x_n \rangle_{n=1}^\infty$ with distinct terms can be chosen.  For the sake of convenience, we call a sequence $\langle z_n \rangle_{n=1}^\infty$ in any non-empty set $T$ to be {\it injective} if $z_n$'s are all distinct.  The question whether an injective sequence can be chosen in Theorem \ref{finitepro} has been considered in the literature. To give an idea, we will give some extracts from \cite{Gili} after presenting some basic concepts and results.

We first state a useful property of semigroups. 

\begin{prop}{\bf(Dichotomy)}. \label{Dichotomy} Let $S$ be a  semigroup. For $m \in S$ either $m^j, \,j =1,2, \ldots$ are all distinct or $m^j$ is an idempotent for some $j \in \mathbb{N}.$
\end{prop} 
We will say $m$ is of {\it infinite order} in the first case and of {\it finite order} in the second case; this is in accordance with the corresponding concepts in group theory.

 Let $(S,\cdot)$ be a semigroup. For $m,n \in S$ we usually write $m\cdot n = mn.$ Let $E(S)$ denote the set of idempotent elements in $S$, i.e., the set of elements $n \in S$ such that $n^2=n.$ We write $E_0(S)=E(S)\backslash \{e\};$ in case, $S$ has the identity $e.$ Let $\widetilde{S}= S \backslash E(S).$ For any element $x \in S$ and any subset $A \subset S,$ we set $x^{-1}A:= \{s \in S: xs \in A\}.$

It follows from Proposition \ref{Dichotomy} that every finite semigroup has an idempotent.

\begin{rem} \label{Remark1}

Clearly, for an element $s \in S$ of infinite order, the set $\{s^n:n \in \mathbb{N}\}$ is a subsemigroup of $S$ and thus $S$ contains a copy of $(\mathbb{N}, +).$ It does happen if  \begin{itemize}
		\item[(a)] $S$ has no idempotent, or
		\item[(b)] if $\widetilde{S}$ is a subsemigroup of $S$ (then it does happen for $\widetilde{S}$ and therefore for $S$ too).
		
		Further, we may note that for $S= (\mathbb{Z}_+,+),$ $\widetilde{S}=\mathbb{N}.$
	\end{itemize}
On the other hand, if $s$ is an element of $S$ of finite order, then $s$ generates a finite subsemigroup of $S$ containing an idempotent. 

\end{rem}

\begin{prop} \cite[Theorem 4.28]{Dona} Let $S$ be an infinite semigroup. Then $S^*:= \beta S \backslash S$ is a subsemigroup of $\beta S$ if and only if for any finite subset $F$ of $S$ and for any infinite subset $A$ of $S$ there exists a finite subset $C$ of $A$ such that $\cap_{x \in C} x^{-1}F$ is finite.
\end{prop}
This is equivalent to saying that $S$ is a moving semigroup defined by Golan and Tsaban \cite[pg. 112]{Gili} in the following way. 

A semigroup $S$ is called {\it moving} if it is infinite and, for each infinite $A \subset S$ and each finite $F \subset S,$ there exist $x_1,x_2, \ldots, x_k \in A$ such that $\{x_1s,x_2s,\ldots, x_ks\} \nsubseteq F$ for all but finitely many $s \in S.$

 It is clear that every right cancellative infinite semigroup is moving. In particular, the semigroup $(\mathbb{Z}_+,+)$ and  every infinite group is moving.

\begin{thm} \label{GGH} (Galvin-Glazer-Hindman) \cite[Theorem 1.2]{Gili}. Let $S$ be a moving semigroup. For each finite colouring of $S,$ there is an injective sequence $\langle x_n \rangle_{n=1}^\infty$ such that  FP$(\langle x_n\rangle_{n=1}^\infty)$ is monochromatic. 
\end{thm} 

\begin{defn} 
	\begin{itemize}
		\item[(i)] An infinite semigroup $S$ is called a {\it Ramsey semigroup} if the conclusion of Galvin-Glazer-Hindman Theorem \ref{GGH} holds for $S$.
		\item[(ii)] A group which is a Ramsey semigroup will be called a {\it Ramsey group}.
	\end{itemize}
	
	\end{defn}
\begin{rem} \label{Smee} \begin{itemize}
		\item[(i)] For a finite subset $F$ of $S,$ we can choose a sequence $\langle x_n \rangle_{n=1}^\infty$ as in Theorem \ref{GGH} to be in $S \backslash F$  by replacing the sequence by a subsequence if needed. 
		\item[(ii)]Let ${\bf x}=\langle x_n\rangle_{n=1}^\infty$ be a sequence in $S$ with infinite range $B.$ Clearly, $S$ is infinite. Then there exists an injective sequence ${\bf y }= \langle y_j\rangle_{j=1}^\infty = \langle x_{n_j}\rangle_{j=1}^\infty,$ a subsequence of ${\bf x},$ so that FP($\langle y_j\rangle_{j=1}^\infty) \subset $ FP$(\langle x_n\rangle_{n=1}^\infty).$ So the requirement that for any partition $\{C_i\}_{i=1}^r$ of $S$ there exist an $i \in \{1,2, \ldots,r\}$ and a sequence $\langle x_n\rangle_{n=1}^\infty$ with infinite range such that FP$(\langle x_n\rangle_{n=1}^\infty)\subset C_i,$ is equivalent to the existence of an $i \in \{1,2, \ldots,r\}$ and an injective  sequence $\langle y_n\rangle_{n=1}^\infty$  such that FP$(\langle y_n\rangle_{n=1}^\infty)\subset C_i.$ 
	\end{itemize}
	
\end{rem}

\begin{exmp} \label{copyN} 
	\begin{itemize}
		\item[(i)] Let $(S,<, \cdot)$ be  an infinite ``max"  semigroup with $m \cdot n= \text{max}\{m,n\}.$ Then $(S,<,\cdot)$ is a Ramsey semigroup.
		\item[(ii)]  Galvin-Glazer-Hindman Theorem \ref{GGH} can be restated as: every moving semigroup is a Ramsey semigroup. In particular, every infinite group is a Ramsey group. 
		\item[(iii)]  If an infinite semigroup $S$ contains a copy of $(\mathbb{N}, +)$ then $S$ is  a Ramsey semigroup. This follows  immediately from the observation that a partition of $S$ induces a partition of $\mathbb{N}$. Further, we can choose a required injective sequence from that copy of $(\mathbb{N},+).$   
		\begin{itemize}
			\item[(a)] 	It follows from Remark \ref{Remark1} (i)(b) that if $(\tilde{S}, \cdot)$ is subsemigroup of $(S, \cdot)$ then $\tilde{S}$ contains a copy of $(\mathbb{N}, +)$ and therefore, both $(S,\cdot)$  and $(\tilde{S}, \cdot)$ are Ramsey semigroups.
			\item[(b)] Thus roughly speaking, it is enough to consider periodic semigroups $S$ in the sense that every element of $S$ has finite order.  
		\end{itemize}
	
		\item[(iv)] It is known that Galvin-Glazer-Hindman Theorem \ref{GGH} cannot be extended to arbitrary semigroups as it is clear from \cite[Example 2.1]{Gili}  given by Golan and Tsaban as follows: Let $k \in \mathbb{N},$ let $S_k:=\{0, 1,2,\ldots, k-1\} \cup k\mathbb{N}+1$ be the commutative semigroup with the operation of addition modulo $k.$ It can be easily seen that $S_k$ is not Ramsey semigroup by assigning each $s \in S_k$ the colour $s  \,\text{mod}\,k.$ 
	\end{itemize} 
	\end{exmp}

\begin{thm}   \label{alram} \cite[Theorem 2.3]{Gili}.
	Let $S$ be an infinite semigroup. For each finite colouring of $S,$ there exist a sequence $\langle x_n \rangle_{n=1}^\infty$ with distinct terms and a finite subset $F$ of  FP$(\langle x_n\rangle_{n=1}^\infty)$ such that  FP$(\langle x_n\rangle_{n=1}^\infty) \backslash F$ is monochromatic.  
\end{thm}

\begin{lem} \cite[Lemma 3.2]{Gili}. For each finite colouring of $\bigoplus_n \mathbb{Z}_2,$ there is an infinite subgroup $H$ of $\bigoplus_n \mathbb{Z}_2$ such that $H \backslash \{0\}$ is monochromatic. 
\end{lem} 
Consequently, for each finite colouring of the group $\bigoplus_n \mathbb{Z}_2$ there is an infinite  almost-monochromatic subgroup $H$ of $\bigoplus_n \mathbb{Z}_2$. Golan and Tsaban \cite{Gili} studied this sort of condition  for semigroups and groups which we give below.

\begin{defn} \begin{itemize}
		\item[(i)] An infinite semigroup $S$ is called an {\it almost-strong Ramsey semigroup} if given any finite colouring of $S$ there exists an infinite almost-monochromatic subsemigroup $T$ of $S.$ 
		\item[(ii)]  An {\it almost-strong Ramsey group} can be defined by replacing semigroup and subsemigroup by group and subgroup respectively in (i) above. 
	\end{itemize}
\end{defn} 
	\begin{exmp} \label{gt}
		\begin{itemize}
			\item[(i)] Consider an infinite semigroup (group) $S$ containing an  infinite subsemigroup (subgroup) $T$ of $S.$ 
		\begin{itemize}
			\item[(a)] If $T$ is an almost-strong Ramsey  semigroup (group) then $S$ is an almost-strong Ramsey semigroup (group). In particular, if $\bigoplus_n \mathbb{Z}_2$ is contained in a group $G$ as subgroup of $G,$  then $G$ is an  almost-strong Ramsey group.
			\item[(b)] If  $S \backslash T$ is finite then the converse of the first part of (a) is true.
		\end{itemize}

			\item[(ii)] Let $k \in \mathbb{N}.$ Consider the commutative semigroup $S_k$ as  in Example \ref{copyN} (iv) above. Then $S_k$  is an almost-strong Ramsey semigroup.   To see this take any colouring $\{A_i\}_{i=1}^r$ of $S.$ Consider the sets $B_i:=A_i \cap (k\mathbb{N}+1), 1 \leq i \leq r.$ Set $\Lambda= \{i \in \{ 1,2, \ldots,r\}: B_i \neq \emptyset\},$ then $\{B_i\}_{i \in \Lambda}$ gives a colouring of $k\mathbb{N}+1.$ Now, note that at least one $B_i$ has to be an infinite subset $C$ of $k\mathbb{N}+1.$ Consider $H= \{0, 1,2,\ldots,k-1\} \cup C.$ Then $H$ is an infinite subsemigroup of $S_k$ and $H \backslash \{0, 1, 2,\ldots,k-1\}$ is monochromatic. Therefore, $H$ is almost-monochromatic. 
			
			\item[(iii)]  It is well know that $(\mathbb{N}, +)$ is not an  almost-strong Ramsey semigroup \cite[Lemma 3.8 (folklore)]{Gili}. It can be seen by considering the 2-colouring of $\mathbb{N}$ given by 
			
			$${\color{red}1} ~{\color{blue}2~ 3}~ {\color{red} 4~ 5~ 6 }~ {\color{blue} 7 ~8~ 9~ 10  }~{\color{red} 11~ 12~ 13~ 14~ 15}~ {\color{blue} 16 ~17~ 18~ 19~ 20~ 21}~ {\color{red} 22 \ldots}$$ where the length of the intervals of elements of identical colours are $1, 2, 3, \ldots$.  
			\item[(iv)] An infinite group $G$ is a {\it Tarski Monster} if, for some large prime $p,$ all proper subgroups of $G$ have cardinality $p.$ Olshanskii \cite{Ola} proved that Tarski Monsters exist for all large enough primes $p$ (see also \cite{Adi}) Clearly, Tarski Monsters are not almost-strong Ramsey groups. 
		\end{itemize}
	\end{exmp}

 \subsection{ Basics of hypergroups}
 Here we come to the relevant basics of hypergroups. We may refer to  \cite{Dunkl} and \cite{Jewett} or any other suitable sources.  
 
 In this paper we are mostly concerned with commutative discrete semiconvos or hypergroups. It is convenient to write the definition in terms of a minimal number of axioms. For instance, see (\cite[Chapter 1]{Lasser}, \cite{Alagh}).

 Let $K$ be a discrete space. Let $M(K)$ be the space of complex-valued regular Borel measures on $K.$ Let $M_F(K)$ and $M_p(K)$ denote the subset of $M(K)$ consisting of measures with finite support and probability measures respectively. Let $M_{F,p}(K)= M_F(K) \cap M_p(K).$ At times, we do not distinguish between $m$ and $\delta_m$ for any $m \in K$ because $m \mapsto \delta_m$ is an embedding from $K$ into $M_p(K).$ Here $\delta_m$ is the unit point mass at $m,$ i.e., the Dirac-delta measure at $m.$ 
  
   We begin with a map $* : K \times K \rightarrow M_{F,p}(K)$. Simple computations enable us to extend `$*$' to a bilinear map called {\it convolution}, denoted by `$*$' again, from $M(K) \times M(K)$ to $M(K).$ At times, for certain $n \in K$ we will write $q_n$ for $ \delta_n* \delta_n$ and $Q_n$ for its support.
   
    A bijective map $\vee: m \mapsto \check{m}$ from $K$ to $K$ is called an {\it involution} if $\check{\check{m}}=m.$ We can extend it to $ M(K)$ in a natural way.  
   
 \begin{defn} \label{semiconvo} A pair $(K,*)$ is called a { \it discrete semiconvo} if the following conditions hold.
 \begin{itemize}
 \item The map $* : K \times K \rightarrow M_{F,p}(K)$ satisfies the associativity condition $$ (\delta_m*\delta_n)*\delta_k = \delta_m* (\delta_n*\delta_k)\,\,\, \text{for all}\, m, n, k \in K.$$
 \item There exists a (necessarily unique) element $e \in K$ such that $$ \delta_m*\delta_e = \delta_e*\delta_m = \delta_m \,\,\,\,\text{for all}\,\, m \in K.$$ 
 \end{itemize}\end{defn}
 A discrete semiconvo $(K,*)$ is called {\it commutative} if $\delta_m*\delta_n= \delta_n*\delta_m$ for all $m,n \in K.$
 
 \begin{defn} A triplet $(K,*, \vee)$ is called a {\it discrete hypergroup} if 
 \begin{itemize}
 \item $(K,*)$ is a discrete semiconvo,
 \item  $\vee$ is an involution on $K$ that satisfies \begin{itemize}
 \item[(i)]  $(\delta_m*\delta_n \check{)}= \delta_{\check{n}}*\delta_{\check{m}}$ for all $m,n \in K$ and 
 \item[(ii)] $e \in \spt(\delta_m*\delta_{\check{n}})$ if and only if $m=n.$
\end{itemize}  
 \end{itemize} \end{defn}
 A discrete hypergroup $(K,*, \vee)$ is called {\it hermitian} if the involution on $K$ is the identity map, i.e., $\check{m}=m$ for all $m \in K.$ 
 
  Note that a hermitian discrete hypergroup is commutative.
  
   We write $(K,*)$ or $(K,*,\vee)$ as $K$ only if no confusion can arise.

 Let K be a commutative discrete hypergroup. For a complex-valued function $\chi$ defined on $K,$  we write $\check{\chi}(m):= \overline{\chi(\check{m})}$ and $\chi(m*n) = \int_K \chi\, d(\delta_m*\delta_n)$ for $m,n \in K.$ Now, define two dual objects of $K:$
 $$ \mathcal{X}_b(K)= \left\lbrace\chi \in \ell^\infty(K): \chi \neq 0, \chi(m*n)= \chi(m) \chi(n) \, \text{for all}\,\, m,n \in K \right\rbrace,$$ $$\widehat{K}= \left\lbrace \chi \in \mathcal{X}_b(K) : \check{\chi}= \chi,\,\mbox{i.e.,}\, \chi(\check{m})= \overline{\chi(m)}\,\, \text{for all} \,m \in K  \right\rbrace.$$ Each $\chi \in \mathcal{X}_b(K)$ is called a {\it character} and each $\chi \in \widehat{K}$ is called a {\it symmetric character}. With the topology of pointwise convergence, $\mathcal{X}_b(K)$ and $\widehat{K}$ become compact Hausdorff spaces. In contrast to the group case, these two dual objects need not be the same and also need not have a hypergroup structure. 

 \begin{defn}{\bf Center of hypergroup.}  Let $K$ be a discrete hypergroup. Dunkl \cite[1.6]{Dunkl} defined the center $Z(K)$ of $K$ as the set of all $x$ in $K$ such the $\spt(\delta_x*\delta_y)$ is a singleton for each $y \in K.$ 
	
	Jewett defined the maximum subgroup of $K$ \cite[10.4]{Jewett}  as the set of all $x$ in $K$ such that $\spt(\delta_x*\delta_{\check{x}})=\{e\}$ and showed that it is exactly same as $Z(K)$ \cite[10.4B]{Jewett}.
\end{defn}
 
 Now we give some examples of hypergroups.
\begin{exmp} {\bf Polynomial hypergroups:} \label{2.4} This is a wide and important class of hermitian discrete hypergroups in which hypergroup structures are defined on $\mathbb{Z}_+.$ This class contains Chebyshev polynomial hypergroups of first kind, Chebyshev polynomial hypergroups of second kind, Gegenbauer hypergroups, Hermite hypergroups,  Jacobi hypergroups, Laguerre hypergroups, etc. For more details see \cite{Dunkl, Lasser1, Bloom,Lasser}. 
 
\begin{thm} \cite[Theorem 5.3, and Lemma 5.1]{Lasser}. \label{gpositive}
	Let ${\bf P}=(P_n)_{n \in \mathbb{Z}_+}$ be an orthogonal polynomial system such that the linearization coefficients $g(n,m;k)$ occurring in $$P_n(x) P_m(x)= \sum_{k=|n-m|}^{n+m} g(n,m;k) P_k(x)$$  satisfy
	$$g(n,m;k) \geq 0,\,\,\,\,\,n,m \in \mathbb{Z}_+, \,\,\,\,\, |n-m| \leq k \leq n+m.$$ Let $*: \mathbb{Z}_+ \times \mathbb{Z}_+ \rightarrow M_{F,p}(\mathbb{Z}_+ )$ be given by $$\delta_n*\delta_m= \sum_{k=|n-m|}^{n+m} g(n,m;k) \delta_k.$$ Then
	\begin{itemize}
		\item[(i)] $K_{\mathbf{P}}=(\mathbb{Z}_+.*)$ is a hermitian discrete hypergroup, called the polynomial hypergroup   (related to $ \mathbf{P}=(P_n)_{n \in \mathbb{Z}_+})$.
		\item[(ii)] $g(n,m;|n-m|)>0$ and $g(n,m;n+m)>0.$
	\end{itemize}
	    
\end{thm}

 This immediately gives the following corollary. 
\begin{cor} \label{subsem}
	For a subhypergroup $(L, *)$ of $(K_{\bf P},*),$  $(L,+)$ is also a subsemigroup of $(\mathbb{Z}_+,+).$
\end{cor} 

For Illustration, we describe CP1, the {\it Chebyshev polynomial hypergroup of first kind} which arises from the Chebyshev polynomials of first kind. In fact, they define the following convolution `$*$' on $\mathbb{Z}_+:$
$$\delta_m*\delta_n= \frac{1}{2} \delta_{|n-m|}+\frac{1}{2} \delta_{n+m} \,\,\,\text{for}\,\, m,n \in \mathbb{Z}_+.$$  For any $k \in \mathbb{N},$ $K= k \mathbb{Z}_+$ with $*|_{K \times K}$ is a discrete hypergroup in its own right. 
\end{exmp}
 The {\it Chebyshev polynomial hypergroup of second kind} $(\mathbb{Z}_+, *),$ say, CP2 arises from the Chebyshev polynomials of second kind and the convolution '$*$' on $\mathbb{Z}_+$ is given by $$ \delta_m*\delta_n= \sum_{k=0}^{\text{min}\{m ,n\}} \frac{|m-n|+2k+1}{(m+1)(n+1)} \delta_{|m-n|+2k}.$$ 
 
  Further, $K= 2 \mathbb{Z}_+$ with $*|_{K \times K}$ is a discrete hypergroup in its own right.
 
\begin{exmp} \label{Dk} {\bf  Dunkl-Ramirez  Discrete hypergroups:}  Let $0 < a \leq \frac{1}{2}.$ Dunkl and Ramirez \cite{Ramirez} defined a convolution structure `$*$' on $\mathbb{Z}_+$ to make it a  hermitian hypergroup $K$. The convolution `$*$' is defined by 
	\begin{equation*}
	\delta_m*\delta_n= \begin{cases}
	\text{max}\{m,n\} & m \neq n, \\  \frac{a^n}{1-a} \delta_0+ \sum_{k=1}^{n-1} a^{n-k} \delta_k+\frac{1-2a}{1-a} \delta_n & m=n \geq 2
	\end{cases}
	\end{equation*} with $\delta_0*\delta_0=\delta_0$ and $\delta_1*\delta_1=\frac{a}{1-a} \delta_0+\frac{1-2a}{1-a} \delta_1.$
	
	The dual $H_a (=\widehat{K})$ of $K$ is a (hermitian) countable compact hypergroup and it can be identified with the one point compactification $\mathbb{Z}_+^*=\{0,1,2, \ldots, \infty\}$ of $\mathbb{Z}_+.$ 
	
	For a prime $p,$ let $\Delta_p$ be the ring of p-adic integers and $\mathcal{W}$ be its group of units, that is , $\{x=x_0+x_1p+ \cdots+ x_np^n+ \cdots \in \Delta_p : x_j = 0,1, \ldots,p-1 \, \text{for}\, j \geq 0 \, \text{and} \, x_0 \neq 0  \}$. For $a=\frac{1}{p},$ $H_a$ derives its structure from $\mathcal{W}$-orbits of action of $\mathcal{W}$ on $\Delta_p$ by multiplication in $\Delta_p.$ Further, in this case we also have $\widehat{H_a}=K.$ 
\end{exmp} 

\begin{exmp} {\bf Hypergroup deformations of semigroups at idempotents}. \label{KHD} Motivated by Discrete Dunkl-Ramirez hypergroups as in Example \ref{Dk} above, the authors attempted to make a ``max" semigroup $(S, <, \cdot)$ with the discrete topology into a hermitian discrete hypergroup by deforming the product on the diagonal. We showed  that this can be done if and only if either $S$ is finite or $S$ is isomorphic to $(\mathbb{Z}_+,<,\text{max}).$  Next, we presented a necessary and sufficient condition for a discrete semigroup to become a semiconvo or hypergroup after deforming the product at idempotents. The details are given in \cite{KKA}; here, we give an idea without proofs.

\begin{itemize}
	\item[(i)] In \cite[Section 3]{KKA}, we deformed the semigroup product of a discrete ``$\text{max}$" semigroup $(S,<, \cdot)$ along the diagonal by replacing the semigroup product `$\cdot$' by convolution product `$*$' as follows:  
	\begin{eqnarray*}
		\delta_m*\delta_n = \delta_n*\delta_m &=& \delta_{m\cdot n} (= \delta_{\text{max}\{m,n\}}) \,\,\,\,\mbox{for}\, m,n \in S\, \text{with}\, m \neq n, \, \text{or},\, m=n=e,\\ \delta_n*\delta_n &=& q_n \,\,\,\,\,\,\,\,\,\,\, \text{for} \, n \in S\backslash \{e\}. 
	\end{eqnarray*}
	Here, $q_n$ is a probability measure on $S$ with finite support $Q_n$ containing $e$ and has the form $\sum_{j \in Q_n} q_n(j) \delta_j$ with $q_n(j)>0$ for $j \in Q_n$ and $\sum_{j \in Q_n} q_n(j)=1.$
	
	We gave a set of necessary and sufficient conditions for $(S,*)$ to become a  hermitian discrete hypergroup.
	
	\begin{thm} \label{Ordered} \cite[Theorem 3.2]{KKA} Let $(S,<, \cdot)$ be a discrete (commutative) ``max" semigroup with identity $e$ and ` $*$' and other related symbols as above. Then $(S,*)$ is a hermitian discrete hypergroup if and only if the following conditions hold. 
		\begin{itemize}
			\item[(i)] Either $S$ is finite or $(S,<,\cdot)$ is isomorphic to $(\mathbb{Z}_+,<, \text{max}).$
			\item[(ii)] For $n \in S  \backslash \{e\},$ we have $\mathcal{L}_n \subset Q_n \subset \mathcal{L}_n \cup \{n\},$ where $\mathcal{L}_n:= \{j \in S: j<n\}.$ 
			\item[(iii)]  If $\#S > 2,$ then for $ e \neq m <n $ in $S,$ we have 
			\begin{itemize}
				\item[(a)] $q_n(e)=q_n(m)q_m(e)$ and 
				\item[(b)] $q_n(e) \left( 1+ \sum_{e \neq k \in \mathcal{L}_n} \frac{1}{q_k(e)} \right) \leq 1;$ 
			\end{itemize} 
			or, equivalently, with $v_n= \frac{1}{q_n(e)}$ for $n \in S,$
			\item[(iii)']  If $\#S > 2,$ then for $ e \neq m <n $ in $S,$ we have 
			\begin{itemize}
				\item[(a)]$q_n(m)= \frac{v_m}{v_n}$ and 
				\item[(b)] $\sum_{k \in \mathcal{L}_n}v_k \leq v_n.$ 
	\end{itemize} \end{itemize} \end{thm} 
\item[(ii)] In \cite[Section 4]{KKA}, we try to make $(S,\cdot)$ into a commutative discrete semiconvo or discrete hypergroup $(S,*)$ by deforming the product on $\mathcal{D}_{E_0(S)}:=\{(m,m): m \in E_0(S)\},$ the diagonal of $E_0(S),$ or the idempotent diagonal of $S,$ say.

Let $(S,\cdot)$ be a semigroup with identity $e$. 
	  A non-empty subset $T$ of $S$ is called an {\it ideal} in $S$ if $TS \subset T$ and $ST \subset T,$ where $TS := \{ts \,: t \in T, s\in S\}$ and similarly for $ST.$ Let $G(S)$ denote the set $\{g \in S: \exists\, h \in S \,\,\text{with}\,\, gh=hg=e\}.$ Then $G(S)$ is a group contained in $S$ called the {\it maximal group}. Set $G_1(S)= \{g \in G(S): \, gm=m \, \text{for all} \, m \in E_0(S)\}.$ Clearly, $G_1(S)$ is a subgroup of $G(S).$  Note that members of $G_1(S)$ act on $E_0(S)$ as the identity via left multiplication of $(S, \cdot).$  $(S, \cdot)$ is called {\it action-free} if $G_1(S)= \{e\}.$

For $n \in E(S),$ let $q_n$ be a probability measure on $S$ with finite support $Q_n$ containing $e.$ We express $q_n= \sum_{j \in Q_n} q_n(j) \delta_j$ with $q_n(j)>0$ for $j \in Q_n$ and $\sum_{j \in Q_n} q_n(j)=1.$ We look for necessary and sufficient conditions on $S$ and $ \{ q_n: n \in E_0(S)\}$ such that $(S, *)$ with `$*$' defined below, is a commutative discrete semiconvo:
\begin{eqnarray*}
	\delta_m*\delta_n =  \delta_n*\delta_m & = &\delta_{mn} \,\,\,\,\mbox{for}\,\, (m,n) \in S\times S \backslash \mathcal{D}_{E_0(S)},   \\
	\delta_n*\delta_n &= &q_n \,\,\,\,\mbox{for}\,\, n \in E_0(S). 
\end{eqnarray*} 
\begin{thm} \label{main} \cite[Theorem 3.8]{KKA} Let $(S,\cdot)$ be a  commutative discrete semigroup with identity $e$ such that $S$ is action-free. Let `$*$' and other related notation and concepts be as above. Then $(S, *)$ is a commutative discrete semiconvo if and only if the following conditions hold.
	\begin{itemize}
		\item[(i)]$E(S)$ is finite or $E(S)$ is isomorphic to $(\mathbb{Z}_+,<,\text{max}),$ where the order on $E(S)$ is defined by $m<n$ if $mn=n \neq m.$
		\item[(ii)] $(\widetilde{S}, \cdot)$ is an ideal of $(S, \cdot).$  
		\item[(iii)] $Q_n \subset E(S)$ for $n \in E_0(S).$
		\item[(iv)] If $n \in E_0(S)$ and $m \in \widetilde{S}$ then $Q_n\cdot m =\{nm\}.$  
		\item[(v)] For $n \in E_0(S),$ we have $\mathcal{L}_n \subset Q_n \subset \mathcal{L}_n \cup \{n\},$  where for $n \in E(S),$ $\mathcal{L}_n := \{j \in E(S): j<n\}.$ 
		\item[(vi)]If $\#E(S) > 2,$ then for $ e \neq m <n $ in $E(S),$ we have the following:
		\begin{itemize}
			\item[($\alpha$)] $q_n(e)=q_n(m) q_m(e)$ and 
			\item[($\beta$)] $q_n(e) \left( 1+ \sum_{e \neq k \in \mathcal{L}_n} \frac{1}{q_k(e)} \right) \leq 1.$
		\end{itemize}
		
	\end{itemize}
	Further, under these conditions, $E(S)$ is a hermitian discrete hypergroup. Moreover, $S$ is a hermitian discrete hypergroup if and only if $S = E(S).$
\end{thm}	
	\end{itemize}
\end{exmp}

\begin{exmp} \label{ocd} Here, we rewrite some of  the results of \cite[Section 8]{Jewett} for the case of discrete groups. 
 
 \begin{defn} \cite[Section 8.1]{Jewett}
 	Let $G$ be a discrete group. A mapping $\varphi: G \rightarrow G $ is called { \it affine} if there exist $y \in G$ and an automorphism $\psi$ of $G$ such that $\varphi(x)=\psi(x)\,y.$ An {\it affine action} of a discrete group $H$ on a discrete group $G$ is an action  $(x,s) \mapsto x^s$ for which each mapping $x \mapsto x^s$ from $G$ to $G$ is affine.  
 \end{defn}

 \begin{itemize}
 	\item[(i)] {\bf Discrete orbit semiconvos}. \cite[Theorem 8.1 B]{Jewett}.  Let $G$ be a discrete group and let $H$ be a finite group with $\#H=c.$ Suppose that $(x,s) \mapsto x^s$ is an affine action of $H$ on $G$. Then the space $G^H $ of orbits $x^H$ given by $x^H= \{x^s: s \in H \}$ equipped with the discrete topology is a semiconvo with respect to the convolution `$*$' defined by 
 	$$ \delta_{x^H}*\delta_{y^H}= \frac{1}{c^2} \sum_{s,t \in H} \delta_{(x^sy^t)^H}.$$ 
 	
 	\item[(ii)] {\bf Discrete coset semiconvos.} \cite[Theorem 8.2 A]{Jewett}.  Let $G$ be a discrete group and let $H$ be a finite subgroup of $G$ with $\#H=c.$  If the action of $H$ on $G$ is given by $(x,s) \mapsto xs$ then it is an affine action and the orbit space $G^H$ is the right coset space $G/H:=\{xH: x \in G\}.$ The space $G/H,$ with the operation `$*$' given by 
 	$$ \delta_{xH}*\delta_{yH}= \frac{1}{c} \sum_{s \in H} \delta_{xsyH}$$ is a semiconvo.
 	
 	\item[(iii)] {\bf Discrete double coset hypergroups.} \cite[Theorem 8.2 B]{Jewett}.  Let $G$ be a discrete group and let $H$ be a finite subgroup of $G$ with $\#H=c.$ If the action of $H \times H$ on $G$ is given by $(x,(s,t)) \mapsto s^{-1}xt$ then it is an affine action and  the orbit space $G^{H \times H}$ is the double coset space $G//H:=\{HxH: x \in G\}.$ The space $G//H,$ with the operation `$*$' given by 
 	$$ \delta_{HxH}*\delta_{HyH}= \frac{1}{c} \sum_{t \in H} \delta_{HxtyH}$$ is a discrete commutative hypergroup with the identity $HeH=H$ and  the involution $(HxH \check{)} = Hx^{-1}H.$
 	\item[(iv)] {\bf Discrete automorphisms orbit hypergroups.} \cite[Theorem 8.3 A]{Jewett}. Let $G$ be a discrete group  and let $H$ be a finite subgroup of the  group of automorphisms of $G$ with $\#H=c.$ Suppose that $(x,s) \mapsto x^s$ is the corresponding  action of $H$ on $G$. Then the space $G^H $ of orbits $x^H$ given by $x^H= \{x^s: s \in H \}$ equipped with the discrete topology is a discrete hypergroup with respect to the convolution `$*$' defined by 
 	$$ \delta_{x^H}*\delta_{y^H}= \frac{1}{c} \sum_{s \in H} \delta_{(x^sy)^H}$$  
 	with the identity $e^H= \text{id}_G$ and the involution $(x^H \check{)}= (x^{-1})^H.$
 	\end{itemize}
	\end{exmp} 
\begin{exmp} \label{Kenthm}   
	
	 Jewett \cite[14.2]{Jewett} proved that if we take $H$ to be a compact subhypergroup of a hypergroup $K$ then the space of double cosets  $K//H$ can be made into a hypergroup. But the second author \cite{Ken} considered a commutative hypergroup $K$ and a closed subgroup $H$ of $Z(K)$ and  defined a different convolution product on K//H to make it a hypergoup. In fact, he studied  many properties of $K//H$ by assuming $K//H$ to be compact. An element of $K//H$ is denoted by $[x],$ the double coset of $x \in K.$ Let $\pi$ denote natural surjection $x \mapsto [x]$ from $K$ onto $K//H.$ 
	   
	   Here, we rewrite some of  the results of \cite[Section 4]{Ken} for the case of a commutative discrete hypergroup $K$ and a subgroup $H$ of $Z(K)$. We refine \cite[Theorem 4.1]{Ken}, whose proof in \cite{Ken} does not require $K//H$ to be compact (finite in our case). 
	
		{\it Let $K$ be a commutative discrete hypergroup and let $H$ be a subgroup of $K.$ Then $K//H$  is a hypergroup with the convolution `$*$' defined by, for $x,y \in K,$ $f$ a complex valued function on $K//H$ vanishing at infinity,}
		\begin{equation} \label{kencon}
		\int_{K//H} f\, d(\delta_{[x]}*\delta_{[y]})= \int_K f\circ \pi(u)\, d(\delta_x*\delta_y)(u).
		\end{equation}
	
	Following \cite{Ken}, we use the following notational conventions: if $x \in K$ and $z \in Z(K),$  the unique element in $\spt(\delta_x*\delta_z)$ is denoted by $xz$ only. Also, as $K$ is commutative, the double cosets in $K//H$ all have the form $xH=\{xz:z \in H\}.$ 
	\end{exmp}

 \section{Ramsey theory for hypergroups} \label{r3}
 We begin with motivation for the development of this paper.
 \subsection{Motivation for Ramsey theory on hypergroups}  We start with polynomial hypergroups as in Example \ref{2.4}.  
 
 \subsubsection{Motivation through polynomial hypergroups} \label{mp} We freely use Example \ref{2.4} . It seems natural that the finite sums in $\mathbb{Z}_+$ will be replaced by the supports of the convolution of finitely many unit point mass measures on $K_{\mathbf{P}}$. 
 
  Let $\mathbf{x}= \langle x_n\rangle_{n=1}^\infty$ be an injective sequence in $\mathbb{Z}_+$ with  range $B.$ For a non-empty finite subset $F$ of $B$ i.e., $F= \{x_{n_j} \,: 1 \leq j \leq m\},$  we set $\delta_F= \delta_{x_{n_1}}*\delta_{x_{n_2}}*\cdot \cdot \cdot*\delta_{x_{n_m}}.$
  
  Let $\{A_i\}_{i=1}^r$ be a partition of $\mathbb{Z}_+.$ We would like that there must be an injective sequence $\langle x_n\rangle_{n=1}^\infty$ with range $B$ and  an $i \in \{1,2, \ldots,r\}$ such that $\sup(\delta_F) \subset A_i$ for all finite subsets $F$ of $B.$
 
 Now, consider the Chebyshev polynomial hypergroup of second kind (CP2) as in Example \ref{2.4}. Take the finite partition  $\{A_{i}\}_{i=1}^3$ where  $A_{i}:=\{n \in \mathbb{Z}_+: n \equiv i-1\,\text{mod} \,3\}.$ Take any injective sequence $\langle x_n\rangle_{n=1}^\infty$ in $\mathbb{N}.$ Since $x_n$'s are distinct, we may take it to be strictly increasing. Then, for $k \in \mathbb{N},$  $1 \leq x_{k}<x_{k}+1 \leq x_{k+1}.$  By choosing $F:=\{x_k, x_{k+1}\},$ we get $\spt(\delta_F) \nsubseteq A_i $ for any $i$ as the support $\spt(\delta_{x_k}*\delta_{x_{k+1}})$ contains two or more elements starting from $x_{k+1}-x_{k}$ to $x_{k+1}+x_k$ with the consecutive differences of $2$ while every $A_i$ contains elements with the difference of multiples of $3.$   Therefore, the situation is different in the setting of hypergroups. It becomes essential and interesting to explore Ramsey theory for hypergroups in detail.
 
 \subsubsection{Motivation through hypergroup deformations of semigroups} \label{3.1.2}  
 It is clear from the definition of the convolution product `$*$' as in Example \ref{KHD}(i) of  hypergroup deformations $(S,*)$ of the semigroup $ (S, \cdot)= (\mathbb{Z}_+,<,\text{max})$ that we can have an analogue of Theorem \ref{finitepro} and Theorem \ref{GGH} for $(S,*)$ with some appropriate changes in notation and terminology as indicated in Subsection \ref{mp} above. We elaborate as follows.  
 
 Let $\mathbf{x}= \langle x_n\rangle_{n=1}^\infty$ be an injective sequence in $\mathbb{N}$ with range $B.$  For a non-empty finite subset $F$ of $B$, i.e., $F= \{x_{n_j} \,: 1 \leq j \leq m\}$, set $\delta_F= \delta_{x_{n_1}}*\delta_{x_{n_2}}*\cdot \cdot \cdot*\delta_{x_{n_m}}.$

 In fact, take a partition $(A_i)_{i=1}^r$ of $\mathbb{Z}_+.$ Then at least one of the $A_i$'s is infinite. In case $A_i$ has identity $e=0$, we replace $A_i$ by $\widetilde{A_i}= A_i\backslash \{e\},$ otherwise we redesignate $A_i$  by $\widetilde{A_i}$. Then $\widetilde{A_i}$ has an injective sequence $\langle x_n\rangle_{n=1}^\infty$ with range $B$ so that all finite products from $\langle x_n\rangle_{n=1}^\infty$ in $(\mathbb{Z}_+,<,\text{max})$ are in $A_i$ by Example \ref{copyN} (i).  Now, for this set and sequence; for any finite subset $F= \{x_{n_j}: 1 \leq j \leq m\}$ of $B,$ $\delta_F$  becomes  $\delta_{y}$ where $y=\underset{1 \leq j \leq m}{\prod} x_{n_j}= \underset{1\leq j \leq m}{\text{max} }x_{n_j}$ and thus $\spt(\delta_F) \subseteq A_i.$\\
 
  This motivates us to study Ramsey theory in the context of general discrete hypergroups or semiconvos.
 
 \subsection{Ramsey principle for hypergroups} \label{4.2} Benjamin Willson \cite{Willson} generalized the concept of colouring in the setting of hypergroups in terms of configurations and  configuration equations. That suggested an idea to formulate the following concepts.
 
 Let $(K, *)$ be an infinite discrete semiconvo. Let $\mathbf{x}= \langle x_n\rangle_{n=1}^\infty$ be an injective sequence in $K \backslash \{e\}.$ We denote its range by $B.$ For a non-empty finite subset $F$ of $B,$ we first write it in its increasing indices form, i.e., $F= \{x_{n_j} \,: 1 \leq j \leq m\}$ with $1 \leq n_1<n_2< \cdots< n_m.$ Next, we set $\delta_F= \delta_{x_{n_1}}*\delta_{x_{n_2}}*\cdot \cdot \cdot*\delta_{x_{n_m}}.$  
 
 Let $\mathbf{x}= \langle x_n\rangle_{n=1}^\infty$ be an injective sequence in $K \backslash \{e\}$ with range $B.$ Set 
 \begin{eqnarray*}
 \text{FC}(\langle x_n\rangle_{n=1}^\infty) &:=& \{\delta_{x_{n_1}}*\delta_{x_{n_2}}*\cdot \cdot \cdot*\delta_{x_{n_m}}: n_1<n_2< \cdots< n_m,\,  m \geq 1 \} \\&=& \{\delta_F: F \,\text{is a non-empty finite subset of}\, B \}
 \end{eqnarray*}  
and 
\begin{eqnarray*}
	\text{SFC}(\langle x_n\rangle_{n=1}^\infty) &:=& \{\spt(\delta_{x_{n_1}}*\delta_{x_{n_2}}*\cdot \cdot \cdot*\delta_{x_{n_m}}): n_1<n_2< \cdots< n_m,\,  m \geq 1 \} \\&=& \{\spt(\delta_F): F \,\text{is a non-empty finite subset of}\, B \}.
\end{eqnarray*}

 \begin{defn} \label{def3.1} (i) Let $(K, *)$ be an infinite discrete semiconvo. $(K,*)$ will be called a  {\it Ramsey semiconvo} if for every partition $K= \bigcup_{i=1}^r C_i,$ there exist $i$ and an injective sequence $\mathbf{x}= \langle x_n \rangle_{n=1}^\infty$ in $K \backslash \{e\}$ such that $\spt(\delta_F) \subset C_i,$ i.e., $\delta_F(C_i)=1$ for every non-empty finite subset $F \subset B.$ In other words, $	\text{SFC}(\langle x_n\rangle_{n=1}^\infty) \subset \mathcal{P}(C_i).$
 	
 	(ii) If $(K,*,\vee)$ is an infinite discrete hypergroup such that $(K,*)$ is a Ramsey semiconvo, then $(K,*,\vee)$ will be called a {\it Ramsey hypergroup}.
 \end{defn} 

\begin{rem} \label{4.3iii} If an infinite discrete subsemiconvo $L$ of a semiconvo $K$ is Ramsey then $K$ is Ramsey. To see this, take a partition  $\{C_i\}_{i=1}^r$ of $K$ and set $\Lambda := \{i \in \{1,2, \ldots,r\}: \widetilde{C_i}=C_i \cap L \neq \phi\}.$ Then $L= \bigcup_{j \in \Lambda} \widetilde{C_j}$ is a partition of $L.$ Since $L$ is Ramsey, there exist an injective sequence $ \mathbf{x}=\langle x_n\rangle_{n=1}^\infty$ in $L \backslash \{e\}$ and $i \in \Lambda$ such that for every non-empty finite subset $F$ of $B$, $ \spt(\delta_F) \subset \widetilde{C_i};$ as $\widetilde{C_i} \subset C_i,$ we obtain that $\spt(\delta_F) \subset C_i.$ Therefore, $K$ is a Ramsey semiconvo.
\end{rem}

\begin{exmp} \label{hysup}  (i). Discussion in \ref{mp} can be summarized as: The Chebyshev polynomial hypergroup of second kind (CP2) is not a Ramsey hypergroup.
	
	(ii). Discussion in \ref{3.1.2} can be summarized as: The hypergroup deformations $(S,*)$ of $(S,\cdot):=(\mathbb{Z}_+,<,\text{max})$ as in Example \ref{KHD} above are all Ramsey hypergroups. In view of Remark \ref{4.3iii}, if a semiconvo or hypergroup $K$ contains a copy of any of these $(S,*)$ then $K$ is a Ramsey semiconvo or hypergroup respectively.
\end{exmp}

\begin{thm} Let $(S,\cdot)$ be an infinite commutative discrete action-free semigroup with the identity $e$ satisfying the conditions (i)-(vi) of Theorem \ref{main}. Then the semiconvo $(S,*)$ is a Ramsey semiconvo.
\end{thm} 
\begin{proof}
	 To prove $(S,*)$ is a Ramsey semiconvo we consider two different cases.
	
	First, consider the case that $E(S)$ is infinite. Then $(E(S),*)$ is a copy of a hypergroup deformation of $(\mathbb{Z}_+,<,\text{max}).$ By Example \ref{hysup} (ii) it follows that $(S,*)$ is a Ramsey semiconvo. 
	
	Finally, we consider the case when $E(S)$ is finite. Then $\widetilde{S} \,(= S \backslash E(S))$ is infinite. Because $S$ satisfies  condition (ii) of Theorem \ref{main}, we have that  $(\widetilde{S},\cdot)$ is a subsemigroup of $(S,\cdot)$ and further, $(\widetilde{S},*)$ coincides with $(\widetilde{S}, \cdot).$ As a consequence, $(\widetilde{S} \cup \{e\},*)$ coincides with $(\widetilde{S} \cup\{e\}, \cdot).$ By Example \ref{copyN} (iii), $(\widetilde{S} \cup \{e\}, \cdot)$ is a Ramsey semigroup with required injective sequence in $\tilde{S}.$ Therefore, $(\widetilde{S}\cup\{e\}, *)$ is Ramsey semoconvo. Now, Remark \ref{4.3iii} implies that $(S,*)$ is a Ramsey semiconvo.
\end{proof}

  For the rest of this subsection we work in the context of Example \ref{Kenthm}.   The following lemma is useful for proving the next few results.
\begin{lem}  \label{3} Let $K$ be a commutative discrete hypergroup and let $H$ be a subgroup of $Z(K).$ Then, for any non-empty subset $E $ of the hypergroup $ K//H$ and $x_1,x_2, \ldots, x_m \in K,$ we have \begin{equation*}
	(\delta_{[x_1]}*\delta_{[x_2]}* \cdots*\delta_{[x_m]})(E)= (\delta_{x_1}*\delta_{x_2}*\cdots*\delta_{x_m})(\pi^{-1}(E)). 
	\end{equation*} 
\end{lem}
\begin{proof} For any non-empty set $E \subset K//H,$ by the definition of `$*$' (see  \eqref{kencon} in Example \ref{Kenthm}), we have, for $x,y \in K,$ 
	$$\int_{K//H} \chi_E \, d(\delta_{[x]}*\delta_{[y]})= \int_K \chi_{\pi^{-1}(E)} (u) \, d(\delta_x*\delta_y)(u)$$ which is equivalent to \begin{equation} \label{fs}
	(\delta_{[x]}*\delta_{[y]})(E)= (\delta_x*\delta_y)(\pi^{-1}(E),\,\, \text{i.e.}, \,\, \chi_E([x]*[y])= \chi_{\pi^{-1}(E)}(x*y).
	\end{equation}
	
	Now, for  $x,y,z \in K,$ we get \begin{eqnarray*}
		(\delta_{[x]}*\delta_{[y]}*\delta_{[z]})(E) &=& \int_{K//H} \chi_E\, d(\delta_{[x]}*\delta_{[y]}*\delta_{[z]})  \\ &=&  \int_{K//H} \int_{K//H} \chi_E([t]*[s])\, d(\delta_{[x]}*\delta_{[y]})([t])\, d(\delta_{[z]})([s]) \\ &=& \int_{K//H} \chi_E([t]*[z])\, d(\delta_{[x]}*\delta_{[y]})([t]) \\ &=& \int_{K//H} \chi_E^{[z]}([t])\, d(\delta_{[x]}*\delta_{[y]})([t]) \\ &=& \int_{K} \chi_E^{[z]}([u])\, d(\delta_x*\delta_y)(u)\,\,\,\,\, \text{[Using  \eqref{kencon}]} \\&=& \int_K \chi_E([u]*[z])\, d(\delta_x*\delta_y)(u)\\ &=& \int_K \chi_{\pi^{-1}(E)}(u*z)\, d(\delta_x*\delta_y)(u) \,\,\,\,\,\,\, [\text{From \eqref{fs}}] \\&=& \int_K \int_K \chi_{\pi^{-1}(E)}(u*v)\, d(\delta_x*\delta_y)\, d(\delta_z)(v) \\ &=& \int_K \chi_{\pi^{-1}(E)}\, d(\delta_x*\delta_y*\delta_z)= (\delta_x*\delta_y*\delta_z)(\pi^{-1}(E)).
	\end{eqnarray*}  
	
	Similarly, if $x_1, x_2, \ldots, x_m \in K,$ then for any $\emptyset \neq  E \subset K//H$ we have 
	\begin{equation*}
	(\delta_{[x_1]}*\delta_{[x_2]}* \cdots*\delta_{[x_m]})(E)= (\delta_{x_1}*\delta_{x_2}*\cdots*\delta_{x_m})(\pi^{-1}(E)). 
	\end{equation*} \end{proof}

\begin{thm}\label{douken} Let $K$ be a commutative discrete hypergroup and let $H$ be a finite subgroup of $Z(K).$ If $K$ is a Ramsey hypergroup then the hypergroup $K//H$ is a Ramsey hypergroup. 
\end{thm}
\begin{proof}
	Take a partition $\{C_i\}_{i=1}^r$ of $K//H$ so that $K//H = \cup_{i=1}^r C_i.$  
	Set $\tilde{C_i}= \pi^{-1}(C_i),\, 1 \leq i \leq r.$ Then $\{\tilde{C_i}\}_{i=1}^r$ is a partition of $K.$ Since $K$ is a Ramsey hypergroup, there exist $i \in \{1,2, \ldots, r\}$ and an injective sequence $\langle x_n \rangle_{n=1}^\infty \subset K\backslash \{e\}$ such that $\delta_{F'}(\tilde{C_i})=1$ for any non-empty finite subset $F'$ of the range of  the sequence $\langle x_n \rangle_{n=1}^\infty.$ For $n \in \mathbb{N},$ set $y_n=\pi(x_n).$  Since $H$ is finite and the $x_n$'s are distinct, we get an injective sequence $\langle y_{n_j}\rangle_{j=1}^\infty,$ subsequence of $\langle y_n \rangle_{n=1}^\infty$ with $n_j'$s strictly increasing. Put $z_j:=y_{n_j}$ for $j \in \mathbb{N}.$ Consider any non-empty finite subset $F$ of the range of $\langle z_j \rangle_{j=1}^\infty$ say, $F:=\{z_{j_k}: 1 \leq k \leq m\}$ with $1 \leq {j_1}<j_2< \cdots < {j_m}.$ Now, set $F':=\{x_{n_{j_k}}: 1 \leq k \leq m\}.$  
	
	Using Lemma \ref{3} , we have 
	\begin{equation} \label{next}
		\delta_F(C_i )= \delta_{F'}(\tilde{C_i}).
	\end{equation} 
	To see this,
	\begin{eqnarray*}
		\delta_F(C_i ) =(\delta_{z_{j_1}}*\delta_{z_{j_2}}*\cdots*\delta_{z_{j_m}})(C_i) &=& (\delta_{y_{n_{j_1}}}*\delta_{y_{n_{j_2}}}* \cdots* \delta_{y_{n_{j_m}}})(C_i)\\ &=& (\delta_{\pi(x_{n_{j_1}})}*\delta_{\pi(x_{n_{j_2}})}*\cdots*\delta_{\pi(x_{n_{j_m}})}) (C_i) \\ &=& \left( \delta_{ \left[ x_{n_{j_1 }}\right]}*\delta_{ \left[ x_{n_{j_2 }}\right]}*\cdots*\delta_{ \left[ x_{n_{j_m }}\right]} \right) (C_i) \\ &=& (\delta_{x_{n_{j_1}}}*\delta_{x_{n_{j_2}}}*\cdots*\delta_{x_{n_{j_m}}}) (\tilde{C_i}) \\ &=& \delta_{F'}(\tilde{C_i}). 
	\end{eqnarray*}

But, $ \delta_{F'}(\tilde{C_i})=1$ and, therefore, $	\delta_F(C_i )=1,$ i.e.,  $\spt(\delta_F) \subset C_i.$ Hence,  $K//H$ is a Ramsey hypergroup. 
\end{proof}

\begin{exmp}
We may take $K=(S,*)$ for any hypergroup deformation of $(\mathbb{Z}_+,<, \text{max})$ with $q_1(1)=0.$ Then $Z(K)=\{0,1\}.$ We take $H=Z(K).$ We note that $$K//H= \{\{0,1\} , \{m\} : m \geq 2\}$$ and  
\begin{eqnarray*}
\delta_{\{m\}}* \delta_{\{n\}} = \begin{cases}
 \delta_{\{\text{max}\{m,n\}\}} & \text{for}\, m \neq n\,\, \text{with}\, m,n \geq 2, \\
 \left(q_m(0)+q_m(1) \right) \delta_{\{0,1\}}+ \sum_{k \geq 2} q_m(k) \delta_{\{k\}} &  \text{for}\,\, m=n \geq 2. 
\end{cases}
\end{eqnarray*}
Then, by Theorem \ref{douken} $K//H$ is a Ramsey hypergroup.
\end{exmp}
\begin{defn} \begin{itemize}
		\item[(i)] Let $(K, *)$ be an infinite discrete semiconvo. $(K,*)$ will be called an {\it almost-Ramsey semiconvo} if for every partition $K= \bigcup_{i=1}^r C_i,$ there exist $i,$ an injective sequence $\mathbf{x}= \langle x_n \rangle_{n=1}^\infty$ in $K \backslash \{e\}$ and a finite subset $\mathcal{F}$ of $\text{SFC}(\langle x_n\rangle_{n=1}^\infty)$ such that $\text{SFC}(\langle x_n\rangle_{n=1}^\infty) \backslash \mathcal{F} \subset \mathcal{P}(C_i).$ 
		\item[(ii)] Let $(K, *)$ be an infinite discrete semiconvo. $(K,*)$ will be called an {\it almost-strong Ramsey semiconvo} if for every partition $K= \bigcup_{i=1}^r C_i,$ there exist an $i \in \{1,2,\ldots, r\},$ an infinite subsemiconvo of $L$ of $K$ and a finite subset $D$ of $L$ such that  $L \backslash D \subset C_i.$ 
		\item[(iii)] If $(K,*,\vee)$ is an infinite discrete hypergroup such that $(K,*)$ is an almost-Ramsey semiconvo then $(K,*,\vee)$ will be called an {\it almost-Ramsey hypergroup}.  An {\it almost-strong Ramsey hypergroup} can be defined by replacing semiconvo and subsemicovo by hypergroup and subhypergroup respectively in (ii) above.  
	\end{itemize}
	
\end{defn}
\begin{rem}
If an infinite discrete subsemiconvo $L$ of a semiconvo $K$ is an almost-Ramsey semiconvo or an almost-strong Ramsey semiconvo,  then $K$ is also an almost-Ramsey semiconvo or an almost-strong Ramsey semiconvo respectively. We have only to modify the proof of Remark \ref{4.3iii}  above to prove these. Clearly, the statement remains true if semiconvo and subsemiconvo are replaced by hypergroup and subhypergroup respectively.
\end{rem} 

\begin{exmp} \begin{itemize}
		\item[(i)]  The Chebyshev polynomial hypergroup of second kind (CP2) as in Example \ref{2.4} is not an almost-Ramsey hypergroup. To see this, we have only to refine the argument in Item \ref{mp}.  To elaborate, take the finite partition  $\{C_{i}\}_{i=1}^3$ of $\mathbb{Z}_+,$ where  $C_{i}:=\{n \in \mathbb{Z}_+: n \equiv i-1\,\text{mod} \,3\}.$ Let, if possible, there exist a (strictly increasing) injective sequence $\langle x_n\rangle_{n=1}^\infty$ in $\mathbb{N},$ $i \in \{1,2,3\}$ and a finite subset $\mathcal{F}$ of $\text{SFC}(\langle x_n\rangle_{n=1}^\infty)$ such that $\text{SFC}(\langle x_n\rangle_{n=1}^\infty) \backslash \mathcal{F} \subset \mathcal{P}(C_i).$ Set $D=\cup \{V:V \in \mathcal{F}\}.$ Then $D$ is a finite subset of $\mathbb{Z}_+$. Now, put $m= \text{max} (D)+1.$  Let $F=\{x_m,x_{2m}\}.$ Then $m \leq x_m <x_m+m \leq x_{2m}$ and $\spt(\delta_F)= \{x_{2m}-x_m+2j: 0 \leq j \leq x_m\}.$ Note that $\spt(\delta_F) \cap D = \emptyset.$  So, $\spt(\delta_F)  \in \text{SFC}(\langle x_n\rangle_{n=1}^\infty) \backslash \mathcal{F} \subset \mathcal{P}(C_i).$ Therefore, $\spt(\delta_F) \subset C_i$ which is not true as  the support $\spt(\delta_F)$ contains two or more consecutive elements of difference $2$ while every $C_i$ contains elements with the difference of multiples of $3.$ 
		\item[(ii)]  The hypergroup deformations $(S,*)$ of $(S,\cdot):=(\mathbb{Z}_+,<,\text{max})$ as in Example \ref{KHD} above are not almost-strong Ramsey hypergroups as $(S,*)$ does not have any proper infinite  subhypergroup.   
	\end{itemize}
\end{exmp}

\begin{thm} 
	Let $K$ be a commutative discrete hypergroup and let $H$ be a finite subgroup of $Z(K).$ If $K$ is an almost-Ramsey hypergroup then the hypergroup $K//H$ is an almost-Ramsey hypergroup. 
\end{thm}
\begin{proof}
	The proof of this theorem follows by using Equation \eqref{next}. 
\end{proof} 
\begin{thm}
	No polynomial hypergroup $K_{\bf P}$ is an almost-strong Ramsey hypergroup.
\end{thm}
\begin{proof}
Let, if possible, $K_{\bf P}$ be an almost-strong Ramsey hypergroup. The 2-colouring of $\mathbb{N}$ used in Example \ref{gt} (iii) provides a clue.  Take the 2-colouring of $\mathbb{Z}_+$ given by 
	$${\color{blue}0}~{\color{red}1} ~{\color{blue}2~ 3}~ {\color{red} 4~ 5~ 6 }~ {\color{blue} 7 ~8~ 9~ 10  }~{\color{red} 11~ 12~ 13~ 14~ 15}~ {\color{blue} 16 ~17~ 18~ 19~ 20~ 21}~ {\color{red} 22 \ldots}$$ where the length of the intervals of elements of identical colours are $1,1, 2, 3,4, \ldots$. Then there exist an $i \in \{ \color{blue}\text{blue}, {\color{red}\text{red}} \},$ an infinite subsemiconvo $L$ of $K$ and a finite subset $D$ of $L$ such that  $L \backslash D \subset C_i.$ By Corollary \ref{subsem}, the underlying set $L$ is also an infinite subsemigroup of $(\mathbb{Z}_+,+).$ But, in view of Example \ref{gt}(iii), this is not possible. Hence, $K_{\bf P}$ is not an almost-strong Ramsey hypergroup.
\end{proof}

\section{Variants of Ramsey principle for hypergroups}  We will freely use notation and terminology from previous Section \ref{r3}. 

 \begin{defn} Let $0\leq \alpha<1.$ 
 	\begin{itemize}
 		\item[(i)] Let $(K, *)$ be an infinite discrete semiconvo .  $(K,*)$ will be called an  {\it $\alpha$-Ramsey semiconvo} if for every partition $K= \bigcup_{i=1}^r C_i,$ there exist $i$ and an injective sequence $\mathbf{x}= \langle x_n \rangle_{n=1}^\infty$ in $K \backslash \{e\}$ such that $\delta_F (C_i) >\alpha$  for every non-empty finite subset $F$ of the range of $\langle x_n \rangle_{n=1}^\infty.$ 
 		\item[(ii)]  If $(K,*,\vee)$ is an infinite discrete hypergroup such that $(K,*)$ is an $\alpha$-Ramsey semiconvo, then $(K,*,\vee)$ will be called an {\it $\alpha$-Ramsey hypergroup}. 
 	\end{itemize}
	
\end{defn} 

\begin{rem} (i) Clearly, a Ramsey semiconvo is $\alpha$-Ramsey for $0\leq \alpha<1.$
	
	(ii) Let $0 \leq \alpha< \beta <1.$ Then a $\beta$-Ramsey semiconvo is $\alpha$-Ramsey. 
	
	(iii) A $0$-Ramsey semiconvo $K$  will be called a {\it recurrent semiconvo}, because it is so, if for every partition $K= \bigcup_{i=1}^r C_i,$ there exist $i$ and an injective sequence $\mathbf{x}= \langle x_n \rangle_{n=1}^\infty$ in $K \backslash \{e\}$ such that $\delta_F (C_i) >0,$ i.e., $\spt(\delta_F) \cap C_i \neq \emptyset$ for every non-empty finite subset $F$ of the range of $\langle x_n \rangle_{n=1}^\infty.$ 

 \end{rem}

\begin{rem} \label{rem4} (i) Let $0 \leq \alpha<1.$ If an infinite discrete sub-semiconvo $L$ of a discrete semiconvo $K$ is $\alpha$-Ramsey then $K$ is $\alpha$-Ramsey. We have only to modify the proof of Remark \ref{4.3iii} after Definition \ref{def3.1} to prove this. Clearly, the statement remains true if the word semiconvo is replaced by hypergroup. 
	
	(ii) Note that if semiconvo $(K,*)$ is a semigroup with identity then Ramsey semiconvo and $\alpha$-Ramsey semiconvo are the same object.
\end{rem}

\begin{thm} \label{ephr}
	Every polynomail hypergroup $K_{\mathbf{P}}$ is a $0$-Ramsey hypergroup, that is, a recurrent hypergroup. 
\end{thm}
\begin{proof} Take a  partition $\{C_i\}_{i=1}^r$ of $\mathbb{Z}_+.$ Then, by Theorem \ref{GGH},  there exist a set $C_i$ and an injective sequence $\langle x_n\rangle_{n=1}^\infty$ in $\mathbb{N}$ such that all the finite sums of $\langle x_n\rangle_{n=1}^\infty$ are in $C_i.$ Since, by Theorem \ref{gpositive} (ii),  $g(n,m,n+m)>0$, it follows that for a finite subset $F$ of $\mathbb{N},$ $\spt(\delta_F)$ contains the element $s_F=\sum_{x_j \in F} x_j$ so that $\spt(\delta_F) \cap C_i \neq \phi,$ i.e.,  $\delta_F(C_i)>0.$ 
\end{proof}

In particular, Chebyshev polynomial hypergroup of second kind (CP2) is a recurrent hypergroup. We have already noted in Example \ref{hysup} (i) that CP2 is not a Ramsey hypergroup. The following theorem shows that CP2 is not even $\alpha$-Ramsey, $0<\alpha<1.$
\begin{thm} \label{CP2notalpha}
	Let $0< \alpha <1.$ Then the Chebyshev polynomial hypergroup of second kind (CP2) is  not an $\alpha$-Ramsey hypergroup. 
\end{thm}
\begin{proof}
 Let, if possible, the hypergroup CP2 $=(\mathbb{Z}_+,*)$ be an $\alpha$-Ramsey hypergroup for some $0<\alpha<1.$ Then for every partition $\mathbb{Z}_+= \bigcup_{i=1}^r C_i,$ there exist $i$ and an injective sequence $\mathbf{x}= \langle x_n \rangle_{n=1}^\infty$ in $C_i \backslash \{e\}$ such that $\delta_F (C_i) >\alpha,$ for any non-empty finite subset $F$ of range $B$ of $\langle x_n \rangle_{n=1}^\infty.$
Consider the partition $\{C_i\}_{i=1}^{4^k}$ of $\mathbb{Z}_+$ with $C_i=\{n: n=(i-1)\,(\text{mod}\, 4^k) \},$ where $k (\geq 2) \in \mathbb{N}$ is such that  $\alpha> \frac{1}{2^{k-1}}.$  Since $\langle x_n \rangle_{n=1}^\infty$ is an injective sequence,  there exist $m,n \in B$ such that $16 \leq 4^k<m<2m<n-m.$ Now, we have 

\begin{equation} \label{4}
\delta_n*\delta_m= \sum_{j=0}^m \frac{n-m+2j+1}{(n+1)(m+1)} \delta_{n-m+2j}.$$ So, $$ \frac{1}{2^{k-1}} < \alpha< (\delta_n*\delta_m)(C_i)= \sum_{\underset{n-m+2j \in C_i}{0\leq j \leq m}} \frac{n-m+2j+1}{(n+1)(m+1)}.
\end{equation}
 
Now, since $m,n \in B \subset C_i,$ we have $m=l_0 4^k+(i-1)$ and $n=l_1 4^k+(i-1)$ for some $l_0,l_1 \in \mathbb{N}.$ Therefore, $n-m= (l_1-l_0)4^k>2m= 2(l_0 4^k+ (i-1)) \geq  2l_0 4^k,$ which, in turn, gives $l_1>3l_0 \geq 3.$ 
  
Let $ 0 \leq j \leq m= l_0 4^k+(i-1).$ Then  $n-m+2j \in C_i$ if and only if $(l_1-l_0) 4^k+2j= (i-1)\,( \text{mod} 4^k).$ In view of \eqref{4} such a $j$ does exist. So $i-1$ must be even, say, $2u.$ As $0 \leq i-1=2u \leq 4^k-1,$ we have, $0 \leq 2u \leq 4^k-2,$ i.e., $0 \leq u \leq \frac{1}{2}4^k-1.$ Further,  
 
 \begin{equation}\label{new4}
 (\delta_n*\delta_m)(C_i)= \sum_{\underset{(l_1-l_0) 4^k+2j= 2u (\text{mod} 4^k)}{0 \leq j \leq l_0 4^k+2u}} \frac{n-m+2j+1}{(n+1)(m+1)}.
 \end{equation} 
  
 Now, for $0 \leq j \leq l_0 4^k+2u,$  either 
 \begin{equation} \label{aA}
 j= 4^k l+t,\,\,\,\,\,\,\, 0 \leq l <l_0, \, 0 \leq t \leq 4^k-1;
 \end{equation}
or 
\begin{equation} \label{Bb}
j=  4^k l_0 +s,\,\,\,\,\,0 \leq s \leq 2u.
\end{equation}
In case of \eqref{aA}, we have 
$n-m+2j= (l_1-l_0) 4^k+2 l 4^k+ 2t,$ which is $2u\,(\text{mod}\, 4^k)$ if and only if $2t=2u\, (\text{mod}\,4^k).$
This holds if and only if either $t=u$ or $t= \frac{1}{2} 4^k+u.$

In case of \eqref{Bb}, we have $n-m+2j= (l_1-l_0)4^k+2 l_0 4^k+2s,$ which is $2u\,(\text{mod}\, 4^k)$ if and only if either $s=u$ or $s= \frac{1}{2} 4^k+u.$ The latter is not possible simply because $\frac{1}{2}4^k+u$ is greater than $2u.$

So, the number of terms summed in \eqref{new4} is $2l_0+1$. Also, each such term  is $\leq \frac{m+n+1}{(m+1)(n+1)} = \frac{(l_04^k+(i-1))+(l_14^k+(i-1))+1}{(l_04^k+(i-1))(l_14^k+(i-1))} < \frac{(l_0+l_1+2)}{l_0l_1 4^k}.$ Thus 

$$ \delta_n*\delta_m(C_i)<\frac{(2l_0+1)(l_0+l_1+2)}{l_0l_1 4^k}.$$ Therefore, by \eqref{4}, we have $\frac{1}{2^{k-1}}<\frac{(2l_0+1)(l_0+l_1+2)}{l_0l_1 4^k},$ which gives $ 2^{k+1} l_0l_1<(2l_0+1)(l_0+l_1+2).$ But $k \geq 2.$ So, by using the inequalities $2+l_0 \leq 3l_0<l_1,$  we get $8l_0l_1<(2l_0+1)(l_0+l_1+2)<(2l_0+1) 2l_1=4l_0l_1+2l_1,$ which, in turn, gives, $4l_0l_1<2l_1,$ i.e., $2l_0<1,$ which is not possible.

Hence, CP2 is not an $\alpha$-Ramsey hypergroup. \end{proof} 

\begin{rem}  Using notation, terminologies and observations made in the proof of  Theorem \ref{CP2notalpha} and doing  direct computations, we can obtain the exact value of $\delta_n*\delta_m(C_i)$ as follows. 
	$$\delta_n*\delta_m(C_i)= \frac{2l_0+1}{m+1}= \frac{2l_0+1}{l_04^k+2u+1}=\frac{2(m-i+1)+4^k}{4^k (m+1)}.$$
	
\end{rem}

	\begin{thm} 
	Let $0 \leq \alpha <1$ and let $K$ be a commutative discrete hypergroup and let $H$ be a finite subgroup of $Z(K).$ If $K$ is an $\alpha$-Ramsey hypergroup then the hypergroup $K//H$ is an $\alpha$-Ramsey hypergroup. 
\end{thm}
\begin{proof}
	The proof of this theorem follows by using \eqref{next} in the proof of Theorem \ref{douken}. 
\end{proof}

Next, we come to semiconvos and hypergroups in Example \ref{ocd}.
\begin{thm} \label{3.6}
	 Let  $G$ be an infinite discrete group and let $H$ be a finite group with $\#H=c.$ Suppose that $(x,s) \mapsto x^s$ is an affine  action of $H$ on $G$. Then the discrete orbit semiconvo $G^H $ is a $0$-Ramsey semiconvo, that is, a recurrent semiconvo. In particular, we have the following facts. 
	\begin{itemize}
		\item[(i)] If $H$ is a finite subgroup of $G$ with $\#H=c$ then the discrete coset semiconvo $G/H$ is a recurrent semiconvo. 
		\item[(ii)] If $H$ is a finite subgroup of $G$ with $\#H=c$ then the discrete double coset hypergroup $G//H$ is a recurrent hypergroup. 
		\item[(iii)] If $H$ is a finite subgroup of the group of automorphisms of $G$ with $\#H=c$ then the discrete automorphism orbit hypergroup $G^H$ is a recurrent hypergroup.  
	\end{itemize}
\end{thm}

\begin{proof} Take any partition $\{C_i\}_{i=1}^r$ of $G^H.$ Set, for $1 \leq i \leq r,\, \widetilde{C_i}= \bigcup_{U \in C_i} U.$ Note that $G = \bigcup_{i=1}^r\widetilde{C_i}$ is a partition of $G.$ Also, for $1 \leq i \leq r$ and $x \in \widetilde{C_i},$ we have $x^H \in C_i.$ Applying Theorem \ref{GGH} to $G$, we get an injective sequence $\langle x_n \rangle_{n=1}^\infty$ in $G \backslash \{e\}$ and $i \in \{1,2,\ldots,r\}$ such that FP$(\langle x_n \rangle_{n=1}^\infty) \subset \widetilde{C_i}.$ For $n \in \mathbb{N},$ set $y_n= x_n^H.$ Because the $x_n$'s are all distinct and $H$ is finite, there exists an injective sequence $\langle y_{n_j}\rangle_{j=1}^\infty,$ a subsequence of $\langle y_{n}\rangle_{n=1}^\infty$. Set $\sigma_j = x_{n_j}$ and $\tau_j= y_{n_j}$ for $j \in \mathbb{N}.$
	
	Consider any non-empty finite subset $F$ of $\mathbb{N},$ say, $F= \{j_k: 1 \leq k \leq m\}$ with $1 \leq j_1<j_2< \cdots< j_m.$ Let $ \sigma_F= \prod_{k=1}^m \sigma_{j_k}$ and $\tau^{}_F=(\sigma_F)^H.$ Then $\sigma_F \in \widetilde{C_i}$ and, therefore, $\tau^{}_F \in C_i.$ Further, for $h= (h_{j_k})_{k=1}^m \in H^m,$ set $\sigma_h= \prod_{k=1}^m \sigma_{j_k}^{h_{j_k}}$ and $\tau_h=(\sigma_h)^H.$ Now, set $F'=\{\tau_{j_k}: 1 \leq k \leq m\}.$ Then $$\delta_{F'}= \delta_{\tau_{j_1}}*\delta_{\tau_{j_2}}* \cdots *\delta_{\tau_{j_m}}= \frac{1}{c^m} \sum_{h \in H^m} \delta_{\tau_h}.$$
	Now, for the identity $e'$ of $H^m,$ we have $\tau_{e'}= \tau^{}_F \in C_i.$ Therefore, $$\delta_{F'}(C_i) \geq \delta_{F'}(\{\tau^{}_{F}\}) \geq \frac{1}{c^m} >0.$$ Hence, $(G^H,*)$ is a $0$-Ramsey semiconvo. 
\end{proof} 

\begin{exmp} \label{ocp}
	Another way to look at CP1 is to think of it as the orbit space of the group $G=\mathbb{Z}$ with the usual addition via the action of the finite group $H=(\{1,-1\},\times )$ of automorphisms of $\mathbb{Z}$ given by multiplication $\times$ as explained in Example \ref{ocd}(iv) and then proceed as in Theorem \ref{ephr} or \ref{3.6} above. 
	For $n \in \mathbb{N}$, let $D_n= \{n, -n\}$ and let $D_0=\{0\}$. Consider any non-empty  finite subset $F=\{n_j: 1 \leq j \leq m\}$ of $\mathbb{Z}_+$. Let $D_F$ be the sum  in $\mathbb{Z}$ of orbits $D_{x_{n_j}},\, 1 \leq j \leq m $. Then $D_F$ turns out to be the union of orbits  $D_k$, $k \in \spt(\delta_F).$  Also it contains the orbit of $s_F= \sum_{j=1}^m x_{n_j}.$
\end{exmp}
\begin{exmp}
	 Let $G$ be the group $\mathbb{Z} \times \mathbb{Z}$ with the usual addition. Let $H := \{id_{\text{G}}, \alpha, \beta
	, \gamma\}$ be the group consisting of automorphisms $\text{id}_G$ and  $\alpha, \beta, \gamma : G \rightarrow G $ given by $\alpha(x,y)=(-x,y), \,\beta (x,y)=(x,-y),\, \gamma(x,y)= (-x,-y).$ By Theorem \ref{3.6}, the hypergroup $G^H$ is a recurrent hypergroup. 
	Another way to see this is to start with the subgroup $M=\mathbb{Z} \times \{0\}$ of $G.$ On $M,$ $H$ reduces to $H_1= \{id_{\text{M}}, \alpha|_{M}\}.$ So by Example \ref{ocp},  $M^{H_1}$ is a recurrent hypergroup. Next, by Remark \ref{rem4} we see that $G^H$ is a recurrent hypergroup.
\end{exmp}

We conclude the string of results by giving a semigroup version of Theorem \ref{3.6} (iii).

\begin{thm} \label{Orbit} 
	Let $(S, \cdot)$ be an infinite discrete Ramsey semigroup with identity $e$. Let $H$ be a finite group of automorphisms of $(S, \cdot)$ with $\#H =c.$ Then the space $S^H$ of orbits $s^H$ given by $s^H =\{\alpha(s): \alpha \in H\},$ equipped with the discrete topology, can be made into a recurrent semiconvo by defining ` $*$' as follows: \begin{eqnarray*}
		\delta_{s^H}* \delta_{t^H}= \frac{1}{c} \sum_{\alpha \in H} \delta_{(\alpha(s)\cdot t)^H}.  
	\end{eqnarray*}
	
\end{thm}
\begin{proof} Simple computations give that `$*$' is associative and $e^H$ works as the identity of $S^H.$ Hence, $(S^H,*)$ is an infinite discrete semiconvo. 
	
	We prove that $(S^H,*)$ is a recurrent semiconvo. A major part of the proof is on lines of that of Theorem \ref{3.6}, but we prefer to give it in full here too. For this take any partition $\{C_i\}_{i=1}^r$ of $S^H.$ Set, for $1 \leq i \leq r,\, \widetilde{C_i}= \bigcup_{U \in C_i} U.$ Note that $S = \bigcup_{i=1}^r\widetilde{C_i}$ is a partition of $S.$ Also, for $1 \leq i \leq r$ and $s \in \widetilde{C_i},$ we have $s^H \in C_i.$ Applying Remark \ref{Smee}(i) to $(S, \cdot)$, we get an injective sequence $\langle s_n \rangle_{n=1}^\infty$ in $S \backslash \{e\}$ and $i \in \{1,2, \ldots,r\}$ such that FP$(\langle s_n \rangle_{n=1}^\infty) \subset \widetilde{C_i}.$ For $n \in \mathbb{N},$ set $t_n= s_n^H.$ Because the $s_n$'s are all distinct and $H$ is finite, there exists an injective sequence $\langle t_{n_j}\rangle_{j=1}^\infty,$ a subsequence of $\langle t_{n}\rangle_{n=1}^\infty$. Set $\sigma_j = s_{n_j}$ and $\tau_j= t_{n_j}$ for $j \in \mathbb{N}.$
	
	Consider any non-empty finite subset $F$ of $\mathbb{N},$ say, $F= \{j_k: 1 \leq k \leq m\}$ with $1 \leq j_1<j_2< \cdots< j_m.$ Let $ \sigma_F= \prod_{k=1}^m \sigma_{j_k}$ and $\tau^{}_F=(\sigma_F)^H.$ Then $\sigma_F \in \widetilde{C_i}$ and, therefore, $\tau^{}_F \in C_i.$ Further, for $\alpha= (\alpha_{j_k})_{k=1}^m \in H^m,$ set $\sigma_\alpha= \prod_{k=1}^m \alpha_{j_k}(\sigma_{j_k})= \alpha_{j_m} \prod_{i=1}^{m} \alpha_{j_m}^{-1} \alpha_{j_k} (\sigma_{j_k})$ and $\tau_\alpha=(\sigma_\alpha)^H.$ Now, set $F'=\{\tau_{j_k}: 1\leq k \leq m\}$ and $\beta'= (\beta, \text{id}_S) \in H^m $ with $\beta=(\beta_1, \beta_2, \ldots, \beta_{m-1}) \in H^{m-1}.$ Then $$\delta_{F'}= \delta_{\tau_{j_1}}*\delta_{\tau_{j_2}}* \cdots *\delta_{\tau_{j_m}}= \frac{1}{c^m} \sum_{\alpha \in H^m} \delta_{\tau_\alpha} = \frac{1}{c^{m-1}} \sum_{\beta \in H^{m-1}} \delta_{\tau_{\beta'}} .$$
	Now, for the identity $\beta_0$ of $H^{m-1},$ we have $\tau_{\beta_0'}= \tau^{}_F \in C_i.$ Therefore, $$\delta_{F'}(C_i) \geq \delta_{F'}(\{\tau_{F}\}) \geq \frac{1}{c^{m-1}} >0.$$ Hence, $(S^H,*)$ is a recurrent semiconvo.
	 \end{proof}
 \begin{exmp}
 	Let $S$ be the group $\mathbb{Z} \times \mathbb{Z}$ or the semigroup $\mathbb{N} \times \mathbb{N}$ with usual addition. Then $H=\{\text{id}_S,\alpha\}$ is a  group of automorphisms of $S$ if $\alpha(x,y)=(y,x)$ for $x,y \in S.$ Then, by Theorem \ref{Orbit}, $S^H$ is a recurrent hypergroup or semiconvo respectively. In fact, it is a Ramsey hypergroup or semiconvo as it contains a copy (coming from the orbits of points on the diagonal) of $(\mathbb{Z},+)$ or $(\mathbb{N}, +)$ respectively. 
 \end{exmp}
 
\begin{exmp} Let $S = (\mathbb{Z} \times \mathbb{N}) \cup \{(0,0)\}$ and  $H_2 =\{\text{id}_S, \alpha|_{S}\},$ where the map $\alpha:S \rightarrow S$ is  given by $\alpha(x,y)=(-x,y)$.  The orbit space $S^{H_2}$ has elements of the following two forms: 
		\begin{itemize}
			\item[(A)]$\{(x,y),(-x,y)\},\,\, 0 \neq x \in \mathbb{Z}$ and $y \in \mathbb{N},$
			\item[(B)]  $\{(0,y)\},\,\, y \in \mathbb{Z}_+.$
		\end{itemize}

		\begin{itemize}
			\item[(a)] We consider  $(S, \cdot)$ with `$\cdot$' derived from the usual addition in $\mathbb{Z} \times \mathbb{Z}.$  Then,  $(S,\cdot)$ is a Ramsey semigroup. Note that $H_2$ is a finite group of automorphisms of $(S, \cdot).$ So we can apply  Theorem \ref{Orbit} above to conclude that $(S^{H_2},*)$ is a recurrent semiconvo. It is clear from (B) that $(S^{H_2},*)$ contains a copy of the semigroup $(\mathbb{Z}_+,+).$ So the conclusion can also be derived by using Remark \ref{rem4}. In fact, we may use Remark \ref{4.3iii} and conclude that $(S^{H_2}, *)$ is Ramsey. 
			\item[(b)] We consider  $(S, \cdot)$ with `$\cdot$' defined by $(x,y) \cdot (x',y')= (x+x', \text{max}\{y,y'\}).$  Then $S$ is an infinite commutative discrete semigroup with identity $e=(0,0).$ Also each element of the form $(x,y)$ with $x \neq 0$ has infinite order. So, by Remark \ref{Remark1} and  Example \ref{copyN}(iii), $(S,\cdot)$ is a Ramsey semigroup.  Further, $H_2$ is a finite group of automorphisms of $(S, \cdot).$ By Theorem \ref{Orbit} above, $(S^{H_2},*)$ is a recurrent discrete semiconvo. In view of (B) above, it contains a copy of the semigroup $(\mathbb{Z}_+, \text{max})$ as a subhypergroup of $(S^{H_2},*).$ Hence, using Remark \ref{Smee}(i) and then Remark \ref{4.3iii} we conclude that $(S^{H_2},*)$ is a Ramsey semiconvo.
		\end{itemize}
	 \end{exmp}
	
\begin{rem} (i) Further variants of these concepts can be given on the lines of concepts in Subsection \ref{4.2} but we do not go into that.
	
	(ii) With a little extra care, the concepts above can be defined for general locally compact semiconvos or hypergroups but we do not go into that. 
\end{rem}

\section*{Acknowledgment}

Vishvesh Kumar thanks the Council of Scientific and Industrial Research, India, for its senior research fellowship. He thanks his supervisors Ritumoni Sarma and N. Shravan Kumar for their support.

A preliminary version of a part of this paper was included in the invited talk by Ajit Iqbal Singh at the conference ``The Stone-$\check{\mbox{C}}$ech compactification : Theory and Applications, at Centre for Mathematical Sciences, University of Cambridge, July 6-8 2016" in honour of Neil Hindman and Dona Strauss. She is grateful to the organizers H.G. Dales and Imre Leader for the kind invitation, hospitality and travel support. She thanks them, Dona Strauss and Neil Hindman and other participants for useful discussion. She expresses her thanks to the Indian National Science Academy for the position of INSA Emeritus Scientist and travel support. Almost half of the contributory talk given by Vishvesh Kumar at the conference ``Abstract Harmonic analysis (AHA)-2018, at National Sun Yat-sen University, Kaohsiung, Taiwan, June 25-29 2018" was based on this paper. He thanks the organizers of the conference and Indian Institute of Technology Delhi for financial support.

The authors thank George Willis for his  useful comment and suggestion.

\end{document}